\documentclass{amsart}

\usepackage{url}

\usepackage{amscd}
\usepackage{amsfonts}
\usepackage{amssymb}
\usepackage{euscript}
\usepackage{amsmath}
\usepackage{amsthm}
\usepackage[matrix, arrow, curve]{xy}
\usepackage{mathrsfs}

\usepackage{tikz}
\usepackage{tikz-cd}
\usetikzlibrary{arrows.meta}
\usepackage{enumitem}

\usepackage{appendix}

\newcommand{\var}{var}
\newcommand{\can}{can}
\newcommand{\Var}{Var}

\theoremstyle{plain}
\newtheorem{prop}{Proposition}
\newtheorem{theorem}{Theorem}
\newtheorem{lemma}{Lemma}

\newtheorem*{theorem*}{Theorem}
\newtheorem*{cor*}{Corollary}

\theoremstyle{definition}
\newtheorem{definition}{Definition}

\theoremstyle{remark}

\begin{document}

\title{Double shuffle relations via the Hodge defect}
\author{Nikita Markarian}

\begin{abstract}
We interpret the convolution constructions of \cite{MarkarianConvolution}
in Hodge-theoretic terms, extending the Hodge defect calculation
of \cite{MarkarianHodge}.
For Hodge--Tate modules on $\mathbb C^*$ smooth outside $1$,
the defect is the action of an element of the Deligne--Terasoma
transport algebra expressed through the Drinfeld associator.
For convolutions of iterated Kummer extensions, comparison of
mixed Hodge structures produces a beta factor and yields the
regularized double shuffle relations.
\end{abstract}

\email{nikita.markarian@gmail.com}

\date{}

\address{UMR 7501, Université de Strasbourg,
7 rue René Descartes,
67084 Strasbourg Cedex, France}
\maketitle

\section*{Introduction}

The regularized double shuffle relations form one of the basic
systems of relations among multiple zeta values.  They combine
the product formulas for their series and iterated-integral
representations \cite{Racinet,IKZ}.  The approach of Deligne and
Terasoma \cite{DTerICM,DeligneTerasoma}, developed in
\cite{EnriquezFurusho}, relates these identities to
multiplicative convolution of perverse sheaves.  A central object
in this approach is the transport algebra $W$, which acts on
vanishing cycles by compositions of variation, transport on
nearby cycles, and the canonical map.  Multiplicative convolution
gives rise to the harmonic coproduct on this algebra.

The paper \cite{MarkarianConvolution} develops this approach by
explicit topological constructions.  Shrinking, or semi-holonomy
along the interval, identifies vanishing cycles with the
cohomology of a suitable extension to $\mathbb P^1$.
Its compatibility with multiplicative convolution relates tensor
products of vanishing cycles to tensor products of cohomology,
and the transport algebra describes the resulting operators.
The purpose of the present paper is to explain the Hodge-theoretic
meaning of these shrinking and convolution constructions.
We equip the relevant objects with mixed Hodge structures and
calculate the periods of the topological comparison maps.

We consider mixed Hodge modules on $\mathbb C^*$ whose restriction
to $X=\mathbb C^*\setminus\{1\}$ is an iterated extension of
constant Tate variations and whose vanishing cycles at $1$ are
Hodge--Tate.  Shrinking along $(0,1]$ identifies vanishing cycles
with the cohomology of the extension by $*$ at $0$ and $!$ at
$\infty$.  Both spaces are Hodge--Tate, and shrinking respects
their weight filtrations.  The splittings supplied by the Hodge
filtrations therefore give a second isomorphism between their
complexifications.  Comparing these two isomorphisms defines
the Hodge defect, which measures the failure of shrinking
to preserve the Hodge filtration.

For iterated Kummer extensions, the Hodge defect was calculated
by convergent integrals in \cite{MarkarianHodge} and expressed
in terms of the Drinfeld associator.  We extend that calculation
to the Hodge modules described above, allowing general gluing data at
$1$.  This extension is needed for the application to convolution.
Theorem~\ref{thm:general-defect-transport} expresses the defect as
the action of a single element of the Deligne--Terasoma transport
algebra, written explicitly through the associator and a local
monodromy factor.  The transport-algebra description of
semi-holonomy in \cite{MarkarianConvolution} thus also describes
its failure to preserve the Hodge filtration.

For the convolution of two iterated Kummer extensions, we compare
the resulting mixed Hodge structure in two ways.  On cohomology,
K\"unneth gives the tensor product of the mixed Hodge structures
of the factors.  On vanishing cycles, we calculate the relative
period matrix of the topological tensor identification of
\cite{MarkarianConvolution}.  The calculation on the exceptional
component of a degeneration of the multiplication fiber gives a
beta function of the two logarithmic monodromies at $1$.
Combining this calculation with the harmonic coproduct gives
Racinet's regularized double shuffle relations for the Drinfeld
associator (Theorem~\ref{thm:double-shuffle}).  In this description,
the beta correction records the failure of the tensor
identification on vanishing cycles to preserve the Hodge
filtration.

Section~1 describes the Hodge--Tate modules under consideration
and proves weight compatibility of shrinking.  Section~2
constructs their logarithmic realization and calculates the Hodge
defect through the transport algebra.  Section~3 treats
convolution, computes the beta factor, and derives the double
shuffle relations using finite quotients of the universal
unipotent path variation.

\medskip
\noindent\textit{Acknowledgements.}
I would like to thank IHES for hospitality and excellent
working conditions.
This project was supported by the PAUSE program and the ITI IRMIA++.

\section{Hodge--Tate modules smooth outside 1}

\subsection{Hodge modules}
\label{subsec:hodge-tate-modules}

Put
\[
X=\mathbb C^*\setminus\{1\}
  =\mathbb P^1\setminus\{0,1,\infty\},
\qquad j_1:X\hookrightarrow\mathbb C^*.
\]
All Hodge structures and Hodge modules have rational coefficients.
The Tate structure $\mathbb Q(n)$ has weight $-2n$.

A graded-polarizable mixed Hodge module $\mathbb M$ on
$\mathbb C^*$ has a rational perverse sheaf $M_B$, a regular
holonomic $\mathcal D_{\mathbb C^*}$-module $\mathcal M$ with a
Hodge filtration $F$, and a comparison
\[
\operatorname{DR}(\mathcal M)\simeq M_B\otimes_{\mathbb Q}\mathbb C.
\]
Both realizations carry a weight filtration $W$, compatible with
this comparison.  These data satisfy the axioms of
\cite[Section~3]{SaitoIntroduction}.  If $\mathbb M$ is smooth on
$X$, its restriction corresponds to an admissible variation of
mixed Hodge structures $\mathbb V$, with underlying local system
$V_{\mathbb Q}$ and
\[
j_1^*M_B=V_{\mathbb Q}[1];
\]
see \cite[Theorem~2.2]{SaitoIntroduction}.  A pure variation of
weight $m$ corresponds to a pure Hodge module of weight $m+1$.

We consider modules smooth on $X$ for which the weight-graded
quotients of $\mathbb V$ are constant Tate variations.  Thus
$\mathbb V$ has Hodge--Tate fibers and a unipotent underlying
local system.
Admissibility is required at $0,1,\infty$: the Hodge filtration
extends to the Deligne extension with locally free double-graded
quotients, and the relative monodromy filtrations exist; see
\cite[(3.13)]{SteenbrinkZucker}.

At $1$ we use the coordinate $1-z$ and denote nearby and vanishing
cycles by $\Psi_1$ and $\Phi_1$, with the perverse shift.  Their
values on a mixed Hodge module are mixed Hodge structures.
For a variation $\mathbb V$, $\Psi_1(\mathbb V)$ denotes the
limiting mixed Hodge structure along the positive real tangent.
We require the vanishing-cycle structure $\Phi_1(\mathbb M)$ to
be Hodge--Tate as well.

For such a variation, $\Psi_1(\mathbb V)$ is Hodge--Tate.
Indeed, monodromy is trivial on the weight-graded pieces, and only
even weights occur.  Thus logarithmic monodromy lowers $W$ by two,
the relative monodromy filtration is $W$ itself, and its graded
limits are the constant Tate quotients.  Write
\begin{equation}
N:\Psi_1(\mathbb V)\longrightarrow\Psi_1(\mathbb V)(-1)
\label{eq:nearby-monodromy}
\end{equation}
for logarithmic monodromy, with its Tate twist.

The gluing theorem \cite[Section~2.4 and Theorem~2.5]{SaitoIntroduction}
gives the following explicit description.  An object $\mathbb M$
with these properties is specified by $\mathbb V$, a
finite-dimensional Hodge--Tate structure $\Phi$, and morphisms
of mixed Hodge structures
\begin{equation}
\begin{tikzcd}[cramped, sep=small]
\Psi_1(\mathbb V) \arrow[r,"\can"]
& \Phi \arrow[r,"\Var"]
& \Psi_1(\mathbb V)(-1)
\end{tikzcd}
\qquad \Var\circ\can=N.
\label{eq:hodge-gluing-data}
\end{equation}
Here $\Phi=\Phi_1(\mathbb M)$ and $\Var$ is
logarithmic variation.  Morphisms are pairs of morphisms on
$\mathbb V$ and $\Phi$ commuting with both arrows.

\begin{definition}
\label{def:hodge-tate-module}
An object $\mathbb M=(\mathbb V,\Phi,\can,\Var)$
as above is called a \emph{Hodge--Tate module smooth outside $1$}
on $\mathbb C^*$.
\end{definition}

\subsection{Shrinking}
\label{subsec:shrinking}

Let $j_0:\mathbb C^*\hookrightarrow\mathbb A^1$ and
$j_\infty:\mathbb A^1\hookrightarrow\mathbb P^1$ be the inclusions, and put
\[
H(\mathbb M)=H^0\!\left(\mathbb P^1,
j_{\infty!}j_{0*}\mathbb M\right).
\]
Direct images here and below are derived.
The functorial mixed Hodge structure on $H(\mathbb M)$ is given by
the direct-image functors for mixed Hodge modules
\cite[Sections~1.2--1.4]{SaitoIntroduction}.
Write $j_{**!}=j_{\infty!}j_{0*}j_{1*}$.
For $\mathbb M=j_{1*}\mathbb V[1]$, one has
$H(\mathbb M)=H^1(j_{**!}\mathbb V)$.
In the following statements, $\Phi=\Phi_1(\mathbb M)$ and
$\Phi_B$ denotes its rational Betti realization.

\begin{prop}[Shrinking]
\label{prop:shrinking}
For every Hodge--Tate module $\mathbb M$ smooth outside $1$,
cohomology $H^k(\mathbb P^1,j_{\infty!}j_{0*}M_B)$ vanishes
for $k\ne0$.
The interval $I=(0,1]$ gives a natural isomorphism
\begin{equation}
\varphi_I:\Phi_B\xrightarrow{\sim}H(\mathbb M)_B.
\label{eq:shrinking}
\end{equation}
Both functors are exact.
\end{prop}

\begin{proof}
Shrinking and cohomological vanishing are
\cite[Proposition~2]{MarkarianConvolution}.  The vanishing and
the long exact cohomology sequence imply exactness of $H$.
Vanishing cycles are exact in the perverse convention used here.
\end{proof}

\begin{prop}[Hodge--Tate property]
\label{prop:hodge-tate-output}
Let $\mathbb M$ be a Hodge--Tate module smooth outside $1$
on $\mathbb C^*$, in the sense of
Definition~\ref{def:hodge-tate-module}.  The mixed Hodge structure
\[
H(\mathbb M)=H^0\!\left(\mathbb P^1,
j_{\infty!}j_{0*}\mathbb M\right)
\]
is Hodge--Tate, and the shrinking isomorphism $\varphi_I$
is strictly compatible with weights.  The functors $H$ and
$\Phi_1$ commute with Tate twists.
\end{prop}

\begin{proof}
By Proposition~\ref{prop:shrinking}, the functors $H$ and
$\Phi_1$ are exact.
Consider the weight filtration of $\mathbb M$ in the category
of mixed Hodge modules.  Every pure quotient
$\operatorname{gr}_m^W\mathbb M$ is a direct sum of constant
Tate Hodge modules on $\mathbb C^*$ and Hodge structures
supported at $1$.  Indeed, pure Hodge modules decompose by
strict support, and a summand with full support is the
intermediate extension of its restriction to $X$;
see \cite[Proposition~1.9 and Theorem~2.3]{SaitoIntroduction}.
That restriction is constant in the present situation.

Both $H$ and $\Phi_1$ vanish on constant Hodge modules on
$\mathbb C^*$ and equal the identity on objects supported
at $1$.  Their values on $\operatorname{gr}_m^W\mathbb M$
are therefore pure of weight $m$.  Consequently,
\[
W_mH(\mathbb M)=H(W_m\mathbb M),\qquad
W_m\Phi_1(\mathbb M)=\Phi_1(W_m\mathbb M).
\]
These graded pieces are Tate because $\Phi_1(\mathbb M)$
is Hodge--Tate.  Thus $H(\mathbb M)$ is Hodge--Tate as well.
Naturality of shrinking with respect to
$W_m\mathbb M\hookrightarrow\mathbb M$ proves strict weight
compatibility.  The constructions commute with Tate twists.
\end{proof}

\section{The Hodge defect}

\subsection{Logarithmic realization}
\label{subsec:logarithmic-realization}

Let $\overline{\mathcal V}$ be the Deligne extension to
$\mathbb P^1$ of the flat bundle of $\mathbb V$.
The extended Hodge filtration splits it into its constant
weight-graded bundles, as in
\cite[Subsection~2.1]{MarkarianHodge}.  Write
$\overline{\mathcal V}=E\otimes\mathcal O_{\mathbb P^1}$,
where $E=\Psi_1(\mathbb V)_{\mathrm{dR}}$ with its Hodge
splitting.  In this trivialization the connection is
\begin{equation}
\nabla=d-N_0\frac{dz}{z}-N_1\frac{dz}{1-z}.
\label{eq:logarithmic-connection}
\end{equation}
The operators $N_0,N_1$ lower weight by two; Griffiths
transversality shows that they have precisely this degree.
In particular, sufficiently long products of them vanish.

Put $c=\can_{\mathrm{dR}}$ and
$v=\Var_{\mathrm{dR}}$ for the components of
\eqref{eq:hodge-gluing-data}.  In matrices we identify the
de Rham Tate line with $\mathbb C$, retaining its filtration
shift $F^p(E(-1))=F^{p-1}E$.  Thus
\begin{equation}
c:E\longrightarrow\Phi_{\mathrm{dR}},\qquad
v:\Phi_{\mathrm{dR}}\longrightarrow E,
\qquad vc=-N_1.
\label{eq:dr-gluing}
\end{equation}
The sign follows from
$\operatorname{Res}_1\nabla=N_1=-N_{\mathrm{dR}}$.
Tate degree $r$ means weight $-2r$.
The map $c$ preserves Tate degree, and $v$ increases it by one.

Write $D=\{0,1,\infty\}$ and $i:\{1\}\hookrightarrow\mathbb P^1$.
The logarithmic complex for $j_{\infty!}j_{0*}\mathbb M$ is
\begin{equation}
\begin{split}
\mathcal K_{\mathbb M}^{-1}
 &=E\otimes\mathcal O_{\mathbb P^1}(-\infty),\\
\mathcal K_{\mathbb M}^{0}
 &=\left\{(\alpha,\xi)\in
 E\otimes\Omega^1_{\mathbb P^1}(\log D)(-\infty)
 \oplus i_*\Phi_{\mathrm{dR}}:
 -\operatorname{Res}_1\alpha=v\xi\right\},\\
d(s)&=(\nabla s,c(s|_1)).
\end{split}
\label{eq:general-logarithmic-complex}
\end{equation}
Here $(-\infty)$ denotes tensor product with
$\mathcal O_{\mathbb P^1}(-\{\infty\})$.
The relation $vc=-N_1$ makes the differential well-defined.
Its $F^p$ part consists of sections with values in $F^pE$
in degree $-1$ and pairs with
$\alpha\in F^{p-1}E\otimes\Omega^1(\log D)(-\infty)$,
$\xi\in F^p\Phi_{\mathrm{dR}}$ in degree $0$.

\begin{prop}
\label{prop:general-logarithmic-realization}
The filtered complex \eqref{eq:general-logarithmic-complex}
computes $H(\mathbb M)_{\mathrm{dR}}$ with its Hodge filtration.
Put $\vartheta=dz/(z(1-z))$.  The map
\begin{equation}
\rho_{\mathbb M}:\Phi_{\mathrm{dR}}
\xrightarrow{\sim}H(\mathbb M)_{\mathrm{dR}},
\qquad \xi\longmapsto[(v\xi\,\vartheta,\xi)]
\label{eq:residue-pair-class}
\end{equation}
is an isomorphism preserving $F$.
\end{prop}

\begin{proof}
For the $*$-extension at $1$, use the filtered logarithmic
complex of \cite[Proposition~2]{MarkarianHodge}.
The derived boundary object $i^*j_{1*}\mathbb V[1]$ is represented by
$[E\xrightarrow{N_1}E]$, in degrees $-1,0$, with filtrations
$F^pE,F^{p-1}E$.  The restriction map is evaluation in
degree $-1$ and residue in degree $0$.
The Milnor triangle represents $i^*\mathbb M$ by
$[E\xrightarrow{c}\Phi_{\mathrm{dR}}]$, and its map to this
boundary complex has components $(1,-v)$.
These identifications respect the Hodge filtrations by the
gluing construction
\cite[Sections~2.4--2.5]{SaitoIntroduction}.

The localization triangle reconstructs $j_{\infty!}j_{0*}\mathbb M$
as the homotopy fiber product of the $*$-extension and
$i_*i^*\mathbb M$ over its restriction at $1$.
Evaluation and residue are surjective on every $F^p$,
so the ordinary fiber product computes this filtered homotopy
fiber product.  It is exactly
\eqref{eq:general-logarithmic-complex}.

Projection of $\mathcal K_{\mathbb M}^0$ onto $i_*\Phi_{\mathrm{dR}}$
is surjective on each $F^p$.  Its kernel is
\[
E\otimes\Omega^1(\log D)(-\infty-\{1\}),
\]
a sum of copies of $\mathcal O_{\mathbb P^1}(-1)$.
This kernel and $\mathcal K_{\mathbb M}^{-1}$, together with
all their Hodge subbundles, have zero cohomology.
Since $\operatorname{Res}_1\vartheta=-1$, their global
cohomology is therefore identified by the lift
$\xi\mapsto(v\xi\,\vartheta,\xi)$, proving the assertion.
\end{proof}

\subsection{The comparison square}
\label{subsec:comparison-square}

For a Hodge--Tate structure $K$, write
\[
s_K^F:\operatorname{gr}^W K_{B,\mathbb C}
\xrightarrow{\sim}K_{B,\mathbb C}
\]
for the splitting supplied by the Hodge filtration: its summand
of weight $2p$ is the inverse of the projection
$F^pK_{B,\mathbb C}\cap W_{2p}K_{B,\mathbb C}
\to\operatorname{gr}_{2p}^W K_{B,\mathbb C}$.
Here $K_{B,\mathbb C}=K_B\otimes\mathbb C$.
Put $H=H(\mathbb M)$ and extend $\varphi_I$ to $\mathbb C$.
Proposition~\ref{prop:hodge-tate-output} gives the square
\begin{equation}
\begin{tikzcd}[cramped, column sep=large, row sep=large]
\operatorname{gr}^W\Phi_{B,\mathbb C}
  \arrow[r,"s_\Phi^F"]
  \arrow[d,"\operatorname{gr}^W\varphi_I"']
&\Phi_{B,\mathbb C}\arrow[d,"\varphi_I"]\\
\operatorname{gr}^W H_{B,\mathbb C}
  \arrow[r,"s_H^F"']
&H_{B,\mathbb C}.
\end{tikzcd}
\label{eq:comparison-square}
\end{equation}
All four arrows are isomorphisms.  Define
\begin{equation}
\mathcal G_I(\mathbb M)=
s_\Phi^F(\operatorname{gr}^W\varphi_I)^{-1}
(s_H^F)^{-1}\varphi_I.
\label{eq:defect-automorphism}
\end{equation}
We call this automorphism of $\Phi_{B,\mathbb C}$ the
\emph{Hodge defect of shrinking along $I$}, or simply the
\emph{Hodge defect}.  It induces the identity on
$\operatorname{gr}^W$ and is itself the identity exactly when
$\varphi_I$ respects the Hodge filtration.

The residue formula in Proposition~\ref{prop:general-logarithmic-realization}
describes the route through the Hodge splittings explicitly.
Indeed, its map $\rho_{\mathbb M}$ preserves $F$ and, by
naturality on $W_m\mathbb M$, also $W$.
On a pure quotient of $\mathbb M$, both $\rho_{\mathbb M}$ and
shrinking are the identity on the summands supported at $1$;
the constant smooth summands contribute zero.  Their maps on
$\operatorname{gr}^W$ therefore agree under Betti--de Rham
comparison.  Since the Hodge splitting is functorial,
$\rho_{\mathbb M}$ is precisely the isomorphism obtained from
$\operatorname{gr}^W\varphi_I$ by these splittings.

Let $c_\Phi:\Phi_{B,\mathbb C}\to\Phi_{\mathrm{dR}}$ and
$c_H:H_{B,\mathbb C}\to H_{\mathrm{dR}}$ be the comparison maps.
Thus the de Rham expression for \eqref{eq:defect-automorphism} is
\begin{equation}
c_\Phi\mathcal G_I(\mathbb M)c_\Phi^{-1}
=\rho_{\mathbb M}^{-1}c_H\varphi_Ic_\Phi^{-1}.
\label{eq:defect-dr-calculation}
\end{equation}
We use the same notation $\mathcal G_I(\mathbb M)$ for this
operator in the Hodge splitting of $\Phi_{\mathrm{dR}}$.
We calculate this comparison using the integral formula of
\cite[Propositions~5--6]{MarkarianHodge} and naturality in the
gluing data, as explained in the next subsection.

\subsection{The associator and the transport algebra}
\label{subsec:defect-transport}

We use the Drinfeld associator \cite[Section~2]{Drinfeld}
with the following convention.  For
\[
dG=\left(e_0\frac{dz}{z}+e_1\frac{dz}{z-1}\right)G,
\]
the solutions $G_0,G_1$ satisfy
$G_0(z)z^{-e_0}\to1$ at $0$ and
$G_1(z)(1-z)^{-e_1}\to1$ at $1$, and
$A=G_1^{-1}G_0$.  If $A_{\mathrm H}$ denotes the associator of
\cite[Subsection~3.2]{MarkarianHodge}, the conventions are related by
\[
A(e_0,e_1)=A_{\mathrm H}(e_0,-e_1).
\]
Denote by $e_1b_1$ the sum of the terms of
$A^{-1}$ beginning with $e_1$, and put
\begin{equation}
U(e_1)=\sum_{m\geq0}\frac{(2\pi i)^m e_1^m}{(m+1)!}.
\label{eq:associator-transport-series}
\end{equation}
Here $e_0,e_1$ are the residue letters at $0,1$.
For the logarithmic connection \eqref{eq:logarithmic-connection},
they are evaluated at $N_0,-N_1$, respectively, since
$dz/(1-z)=-dz/(z-1)$.  Local Betti monodromy at $1$ is
$\exp(2\pi i e_1)$ in these coordinates.

We use the logarithmic transport algebra of
\cite[Subsection~5.3]{MarkarianConvolution}, with a unit adjoined
and complex coefficients.  In the coordinates of the
differential equation, it has the presentation
\begin{equation}
W=
\mathbb C\cdot1\oplus
\mathbb C\langle\!\langle e_0,e_1\rangle\!\rangle e_1.
\label{eq:dr-transport-algebra}
\end{equation}
All completions are taken with respect to total degree in $e_0,e_1$.
Its action is obtained by composing logarithmic
variation, transport on nearby cycles, and $\can$.
We specify the coordinates by their action: if $\operatorname{rev}$
reverses words, the right action of $ae_1$ on $\Phi_{\mathrm{dR}}$ is
the column operator
\begin{equation}
R(ae_1)=c\,\operatorname{rev}(a)(N_0,-N_1)\,v,
\qquad R(1)=1.
\label{eq:dr-transport-action}
\end{equation}
Indeed, $vc=-N_1$ gives $R(xy)=R(y)R(x)$, as required for a
right action.  All series act by finite sums.

For a series $f=1+f_0e_0+f_1e_1$, put $f_Y=1+f_1e_1$.

\begin{theorem}
\label{thm:general-defect-transport}
Let $\mathbb M$ be a Hodge--Tate module smooth outside $1$
on $\mathbb C^*$, in the sense of
Definition~\ref{def:hodge-tate-module}.  Its Hodge defect is the action of
\begin{equation}
g_I=U(e_1)A(-e_0,-e_1)_Y\in W,
\qquad \mathcal G_I(\mathbb M)=R(g_I).
\label{eq:general-defect-transport}
\end{equation}
More explicitly, in the Hodge frame of $\Phi$ one has
\begin{equation}
\mathcal G_I(\mathbb M)=
1+c\left(
\frac{U(e_1)-1}{e_1}+b_1U(e_1)
\right)(N_0,-N_1)v.
\label{eq:general-defect-explicit}
\end{equation}
The quotient in this formula denotes the power series obtained
by deleting one $e_1$ from each term of $U(e_1)-1$.
The coefficients of $g_I$ belong to the algebra generated by
multiple zeta values and $2\pi i$.
\end{theorem}

\begin{proof}
We first explain why the calculation for the $*$-extension
determines the answer for arbitrary gluing data.  The logarithmic
complex, its local Betti comparison, and interval shrinking are
defined for finite complex gluing data
$(E,\Phi,N_0,c,v)$, with $N_1=-vc$, for which all paths of
sufficiently large degree act by zero; the degree is specified below.
No rational structure is needed for this analytic comparison.
Let $S_{\mathrm{dR}}$ be shrinking transported through the complex
Riemann--Hilbert comparison.  The residue-pair map
$\rho_{\mathbb M}$ is defined by the same formula for these
complex data, so $\rho_{\mathbb M}^{-1}S_{\mathrm{dR}}$ defines
a natural operator on them.  For Hodge modules this is exactly
\eqref{eq:defect-dr-calculation}.
The local comparison uses $G_1$ on
nearby cycles and the topological gluing maps $c$ and
$2\pi i\,U(-N_1)v$; their composite is
$\exp(-2\pi iN_1)-1$.
Neither these data nor their morphisms need be graded.

Here is the precise algebraic consequence of naturality.
Consider the path algebra with vertices $E,\Phi$, a loop $N_0$
at $E$, and arrows $c:E\to\Phi$, $v:\Phi\to E$.
Give $N_0,v$ degree $1$ and $c$ degree $0$, and complete by
degree.  Its quotients by paths of degree greater than $m$ are
finite-dimensional.  On each quotient, natural endomorphisms
of the functor taking a representation to its $\Phi$-space
are exactly the linear combinations of paths starting and
ending at $\Phi$: evaluate on the corresponding regular
projective module, which represents this functor.
Compatibility with the quotients gives the same assertion for
the completion.  Every path from $\Phi$ to itself, other than
the constant path, is uniquely of the form
$c\,h(N_0,vc)\,v$.
On an object supported at $1$, shrinking is the identity.
Consequently the universal comparison operator has the form
$1+c\,h(N_0,-N_1)\,v$ for a noncommutative series $h$.

For the $*$-extension, the logarithmic variation frame gives
$v=1$ and $c=-N_1$, which is the evaluation of $e_1$.
This specialization is injective on the corner series: it
takes $1+chv$ to $1+e_1h$ in the free power-series algebra.
The calculation in \cite[Theorem~1]{MarkarianHodge} identifies the latter
operator with $(1+e_1b_1)U(e_1)$.  Hence
\[
h=\frac{U(e_1)-1}{e_1}+b_1U(e_1),
\]
which proves \eqref{eq:general-defect-explicit}.

The group-like property of the associator gives
\[
\operatorname{rev}(A^{-1})=A(-e_0,-e_1).
\]
Reversing $(1+e_1b_1)U(e_1)$ therefore gives
$U(e_1)A(-e_0,-e_1)_Y$.  Formula
\eqref{eq:dr-transport-action} now identifies
\eqref{eq:general-defect-explicit} with
\eqref{eq:general-defect-transport}.
The coefficient assertion follows from the corresponding
assertion for $A$ and the displayed formula for $U$.
\end{proof}

\section{Double shuffle relations}

\subsection{Multiplicative convolution}
\label{subsec:multiplicative-convolution}

Let $\mathbb V_1,\mathbb V_2$ be two iterated Kummer extensions
on $X$ as in \cite[Subsection~1.1]{MarkarianHodge}, and put
\[
\mathbb E=j_{**!}\mathbb V_1[1],\qquad
\mathbb F=j_{**!}\mathbb V_2[1].
\]
Recall the compactification of multiplication from
\cite[Subsections~2.1--2.2]{MarkarianConvolution}.
Let $S$ be the blow-up of $\mathbb P^1\times\mathbb P^1$ at
$(0,0)$ and $(\infty,\infty)$.  In coordinates $(x,y)$, the maps
$p_5,p_4,p_3:S\to\mathbb P^1$ are respectively the extensions
of $x,y,y/x$.  Write $\iota(z)=z^{-1}$.
Define the convolution by the full derived direct image
\begin{equation}
\mathbb E\ast\mathbb F
:=\mathbf R p_{3*}\!\left(
{p_5}^*\iota^*\mathbb E\otimes
{p_4}^*\mathbb F\right).
\label{eq:compactified-convolution}
\end{equation}
The construction takes place in the derived category of
mixed Hodge modules.

\begin{prop}
\label{prop:convolution-hodge-module}
The restriction $\mathbb M=(\mathbb E\ast\mathbb F)|_{\mathbb C^*}$ is a
Hodge--Tate module smooth outside $1$, in the sense of
Definition~\ref{def:hodge-tate-module}.  In particular, its
underlying perverse sheaf restricts to a shifted local system
on $X$.  There is a canonical isomorphism
\begin{equation}
\mathbb E\ast\mathbb F\simeq j_{\infty!}j_{0*}\mathbb M.
\label{eq:convolution-boundary-extension}
\end{equation}
\end{prop}

\begin{proof}
On $\mathbb C^*$, the construction is multiplicative convolution:
\[
\mathbb M=\mathbf Rm_!\!\left(
j_{1*}\mathbb V_1[1]\boxtimes j_{1*}\mathbb V_2[1]\right),
\qquad m(u,v)=uv.
\]
This is the fiber calculation in
\cite[proof of Proposition~6]{MarkarianConvolution}.
The functor $j_{1*}$ is exact on shifted variations, and the
class of Definition~\ref{def:hodge-tate-module} is closed under
extensions.  Applying the Tate filtrations of the two factors
therefore reduces the first assertion to
$\mathbb V_1=\mathbb V_2=\mathbb Q_X$, up to Tate twists.

Put $\mathbb L=\mathbb Q_{\mathbb C^*}[1]$ and
$\mathbb K=j_{1*}\mathbb Q_X[1]$, and let
$i_1:\{1\}\hookrightarrow\mathbb C^*$ be the inclusion.
The residue sequence is an exact sequence of mixed Hodge modules
\[
0\longrightarrow \mathbb L\longrightarrow \mathbb K
\longrightarrow i_{1*}\mathbb Q(-1)\longrightarrow0.
\]
It gives $\mathbf R\Gamma_c(\mathbb C^*,\mathbb K)\simeq\mathbb Q(0)$
in degree zero: the connecting map
$\mathbb Q(-1)\to H_c^2(\mathbb C^*,\mathbb Q)$ is the
residue isomorphism.  Changing coordinates from $(u,v)$ to
$(uv,v)$ consequently gives $\mathbf Rm_!(\mathbb L\boxtimes \mathbb K)=\mathbb L$.
Convolving the residue sequence with $\mathbb K$ yields a triangle
whose outer terms are $\mathbb L$ and $\mathbb K(-1)$.  Hence its middle term
is a mixed Hodge module, with an exact sequence
\[
0\longrightarrow \mathbb L\longrightarrow\mathbf Rm_!(\mathbb K\boxtimes \mathbb K)
\longrightarrow \mathbb K(-1)\longrightarrow0.
\]
On $X$ it is an extension of the constant Tate variations
$\mathbb Q_X(-1)$ and $\mathbb Q_X(0)$.  Exactness of vanishing
cycles gives
$\Phi_1(\mathbf Rm_!(\mathbb K\boxtimes \mathbb K))=\mathbb Q(-2)$,
since $\Phi_1(\mathbb L)=0$ and $\Phi_1(\mathbb K)=\mathbb Q(-1)$.
This proves the first assertion and justifies using the
untruncated direct image in \eqref{eq:compactified-convolution}.

For the boundary conditions, the calculation in
\cite[proof of Proposition~6]{MarkarianConvolution} applies
before truncation.  Over $\infty$, the coefficient complex
restricts to zero on both components of the fiber of $p_3$,
because one of its factors is extended by $!$.
Over $0$, it is a $*$-extension across both components;
near an exceptional point, coordinates $y=xt$ give a
$*$-extension across $t=0$ and a $!$-extension across $x=0$.
These extensions across transverse divisors commute.
Proper base change therefore gives zero stalk at
$\infty$ and zero costalk at $0$ for $\mathbb E\ast\mathbb F$.
These are precisely the conditions in
\eqref{eq:convolution-boundary-extension}.
\end{proof}

\begin{prop}
\label{prop:convolution-cohomology}
The groups $H^k(\mathbb P^1,(\mathbb E\ast\mathbb F)[-1])$ vanish for $k\ne1$.
There is a canonical isomorphism of mixed Hodge structures
\begin{equation}
H^1\!\left(\mathbb P^1,(\mathbb E\ast\mathbb F)[-1]\right)
\simeq
H^1(j_{**!}\mathbb V_1)\otimes
H^1(j_{**!}\mathbb V_2).
\label{eq:convolution-kunneth}
\end{equation}
\end{prop}

\begin{proof}
The cohomological vanishing follows from
\eqref{eq:convolution-boundary-extension} and
Proposition~\ref{prop:shrinking}.
For the isomorphism, we use the argument of
\cite[Proposition~7]{MarkarianConvolution} in mixed Hodge
modules.  Let $b:S\to\mathbb P^1\times\mathbb P^1$ be the
blow-up map and put
$\mathbb B=\iota^*\mathbb E\boxtimes\mathbb F$.
The complex inside the direct image in
\eqref{eq:compactified-convolution} is $b^*\mathbb B$.
The derived stalk of $\mathbb B$ vanishes at both blow-up
centers: $\iota^*\mathbb E$ is extended by $!$ at $0$,
and $\mathbb F$ is extended by $!$ at $\infty$.
Thus proper base change shows that the adjunction
$\mathbb B\to\mathbf Rb_*b^*\mathbb B$ is an isomorphism.
K\"unneth now gives
\[
\mathbf R\Gamma(\mathbb P^1,\mathbb E\ast\mathbb F)
\simeq
\mathbf R\Gamma(\mathbb P^1,\mathbb E)\otimes
\mathbf R\Gamma(\mathbb P^1,\mathbb F).
\]
All maps are induced by the operations on mixed Hodge modules,
so this is an isomorphism in the derived category of mixed
Hodge structures.  Taking degree zero and removing the
perverse shifts gives \eqref{eq:convolution-kunneth}.
\end{proof}

\subsection{Beta factor}
\label{subsec:beta-factor}

We calculate the mixed Hodge structure on
$\Phi_1(\mathbb E\ast\mathbb F)$ using the exceptional component
of the fiber over $1$.  The construction of
\cite[Subsection~3.3]{MarkarianConvolution} gives an isomorphism
of rational vector spaces
\begin{equation}
\kappa_B:\Phi_1(\mathbb E)_B\otimes\Phi_1(\mathbb F)_B
\xrightarrow{\sim}\Phi_1(\mathbb E\ast\mathbb F)_B.
\label{eq:convolution-vanishing-betti}
\end{equation}
The period calculation measures its compatibility with the
Hodge splittings.

Put $E_i=\Psi_1(\mathbb V_i)_{\mathrm{dR}}$, and write the
connection of $\mathbb V_i$ as in
\eqref{eq:logarithmic-connection}, with operators
$N_0^{(i)},N_1^{(i)}$.
On $E_1\otimes E_2$ put
\[
a=N_1^{(1)}\otimes1,\qquad b=1\otimes N_1^{(2)}.
\]
These operators commute and are nilpotent.
Logarithmic variation identifies the de Rham realizations of
$\Phi_1(\mathbb E)$ and $\Phi_1(\mathbb F)$ with $E_1(-1)$ and
$E_2(-1)$, respectively.
We use these identifications and suppress Tate twists only in
matrix formulas.

Blow up $S$ once more at $(x,y)=(1,1)$.  With $z=y/x$, use the
coordinates
\[
h=1-z,\qquad t=\frac{1-y}{h}.
\]
The exceptional component is $\mathbb P^1_t$; its intersection
with the other component of the fiber is $t=\infty$.
Since
\[
1-y=ht,\qquad 1-\frac1x=\frac{h(1-t)}{1-ht},
\]
the two singular points on it are $t=0$ and $t=1$.
Let $\mathbb T$ be the tensor product of the logarithmic
specializations of the two factors at $1$, pulled back by
$1-t$ and $t$, respectively.  On
$\mathbb P^1_t\setminus\{0,1,\infty\}$ its logarithmic realization is
\begin{equation}
\nabla_{\mathbb T}=d+b\frac{dt}{t}-a\frac{dt}{1-t},
\qquad G(t)=t^{-b}(1-t)^{-a}.
\label{eq:exceptional-component-connection}
\end{equation}
Here the Hodge filtration is the constant tensor-product
filtration on $E_1\otimes E_2$.  Its Betti realization is the
product of the Verdier specializations of the two factors,
with the first pulled back by $t\mapsto1-t$.  On $0<t<1$ the
powers in \eqref{eq:exceptional-component-connection} use real
logarithms.  Extend $\mathbb T$ by $*$ at $0,1$ and then by
zero across the node $\infty$, and put
\[
H_{\mathbb T}=H^1(\mathbb P^1_t,j_{**!}\mathbb T).
\]

\begin{lemma}
\label{lem:exceptional-component-mhs}
The exceptional component gives an isomorphism of mixed Hodge
structures
\begin{equation}
\lambda:H_{\mathbb T}(-1)
\xrightarrow{\sim}\Phi_1(\mathbb E\ast\mathbb F).
\label{eq:exceptional-component-mhs}
\end{equation}
\end{lemma}

\begin{proof}
Let $\beta:\widetilde S\to S$ be the blow-up at $p=(1,1)$,
and denote the coefficient complex in
\eqref{eq:compactified-convolution} by $\mathcal A$.
Write $C=p_3^{-1}(1)$ and
$\widetilde C=(p_3\beta)^{-1}(1)$.
The exceptional component minus its node is $\mathbb A^1_t$;
let $k:\mathbb A^1_t\hookrightarrow\widetilde C$ and
$j_t:\mathbb A^1_t\setminus\{0,1\}\hookrightarrow\mathbb A^1_t$
be the inclusions.  We use perverse nearby cycles
$\Psi_h=\psi_{h,1}[-1]$ and put
$\mathcal N=\Psi_h(\beta^*\mathcal A)$.

The logarithmic description of nearby cycles gives an
identification in mixed Hodge modules
\[
k^*\mathcal N\simeq j_{t*}\mathbb T[1].
\]
Indeed, the substitutions $1-y=ht$ and
$1-1/x=h(1-t)/(1-ht)$ give
\eqref{eq:exceptional-component-connection} on $h=0$,
with the tensor-product limiting Hodge filtration.  The
holomorphic factors of the normalized local fundamental
solutions tend to the identity as $h\to0$, so the remaining
$t$-dependence is precisely $t^{-b}(1-t)^{-a}$.
At $t=0,1$, the divisors have normal crossings; taking
nearby cycles in $h$ retains the $*$ condition at the other
divisor.  The normal monodromy logarithm is $-(a+b)$.
It lowers the coefficient weight filtration by two, so that
filtration is its relative monodromy filtration as well.
These descriptions of the limiting filtrations are those of
\cite[(3.13)]{SteenbrinkZucker}; see also
\cite[Sections~3.2--3.3]{SaitoIntroduction}.

Apply the adjunction $k_!k^*\mathcal N\to\mathcal N$.
Let $\beta_0:\widetilde C\to C$ be the restriction of $\beta$.
The adjunction $\mathcal A\to\mathbf R\beta_*\beta^*\mathcal A$
is an isomorphism over $h\ne0$.  Nearby cycles depend only on
that restriction, so proper base change gives
\[
\mathbf R\beta_{0*}\mathcal N
\simeq\Psi_h(\mathbf R\beta_*\beta^*\mathcal A)
\simeq\Psi_h(\mathcal A).
\]
Taking degree-zero cohomology therefore gives a morphism of
mixed Hodge structures
\[
i_{\mathrm{exc}}:H_{\mathbb T}\longrightarrow
\Psi_1(\mathbb E\ast\mathbb F).
\]
We verify that its image lies in the image of variation.
Let $i:C\hookrightarrow S$ and $s:\{p\}\hookrightarrow C$.
After direct image to $C$, the exceptional supported
object is $s_*H_{\mathbb T}$, concentrated in degree zero
by shrinking.  On rational complexes, the Milnor boundary is
\[
\Psi_h(\mathcal A)\longrightarrow i^!\mathcal A[1].
\]
By adjunction, its composite with the supported map corresponds
to a morphism
$H_{\mathbb T}\to s^!i^!\mathcal A[1]=0$.
The target vanishes because, near $p$, $\mathcal A$
is the external product of the two $*$-extensions, whose
costalks at $1$ vanish.  The Milnor triangle therefore shows that
the displayed map lands in the image of topological variation.

The convolution has no subobject supported at $1$.
The Tate filtrations reduce this assertion to the constant
case of Proposition~\ref{prop:convolution-hodge-module}:
neither $\mathbb L$ nor $\mathbb K(-1)$ has such a subobject,
and the property is preserved by extensions.
Thus logarithmic variation is injective.  It has the same
image as topological variation, and the displayed map,
after the twist $(-1)$, factors uniquely as
\[
H_{\mathbb T}(-1)\xrightarrow{\lambda}
\Phi_1(\mathbb E\ast\mathbb F)
\xrightarrow{\Var}
\Psi_1(\mathbb E\ast\mathbb F)(-1).
\]
Both the original map and logarithmic variation are morphisms
of mixed Hodge structures, so $\lambda$ is one as well.
By construction, $\Var\circ\lambda=i_{\mathrm{exc}}(-1)$.

Both sides of $\lambda$ are exact in each input variation,
by nearby cycles, shrinking, and
Proposition~\ref{prop:convolution-hodge-module}.
Their Tate filtrations therefore reduce invertibility to
constant inputs.  In that case both sides are $\mathbb Q(-2)$,
and the map is nonzero: with $u=1/x=z/(1-ht)$, the exceptional
residue class becomes on a nearby multiplication fiber
\[
\frac{dt}{t(1-t)}
=\frac{du}{u-z}-\frac{du}{u-1}.
\]
Its residues at the two distinct punctures are $1$ and $-1$,
so its class is nonzero.  This proves
\eqref{eq:exceptional-component-mhs}.
\end{proof}

The oriented residue frame
$w\mapsto[w\,dt/(t(1-t))]$ and logarithmic variation on the two
factors give, through $\lambda$, an isomorphism
\[
\kappa_F:
\Phi_1(\mathbb E)_{\mathrm{dR}}\otimes
\Phi_1(\mathbb F)_{\mathrm{dR}}
\xrightarrow{\sim}\Phi_1(\mathbb E\ast\mathbb F)_{\mathrm{dR}}.
\]
It preserves $F$ and $W$ by the residue description of
$H_{\mathbb T,\mathrm{dR}}$ in
\cite[Subsection~2.1]{MarkarianHodge}, applied to
\eqref{eq:exceptional-component-connection}.
Let $c_{\mathbb E},c_{\mathbb F},c_{\mathbb E\ast\mathbb F}$
denote the Betti--de Rham comparisons on vanishing cycles.

\begin{prop}[Beta factor]
\label{prop:beta-factor}
The maps $\kappa_F$ and $\kappa_B$ agree on associated graded
spaces under Betti--de Rham comparison.
The relative period matrix is
\begin{equation}
(c_{\mathbb E}\otimes c_{\mathbb F})\kappa_B^{-1}
 c_{\mathbb E\ast\mathbb F}^{-1}\kappa_F
=\mathcal B(a,b),
\label{eq:convolution-relative-period}
\end{equation}
where
\begin{equation}
\begin{split}
\mathcal B(a,b)
&=(1+a+b)\int_0^1 t^b(1-t)^a\,dt\\
&=\frac{\Gamma(1+a)\Gamma(1+b)}{\Gamma(1+a+b)}.
\end{split}
\label{eq:beta-factor}
\end{equation}
The integral converges entrywise, and the gamma functions are
evaluated by their Taylor series at $1$.
\end{prop}

\begin{proof}
Recall the power series
$U$ of \eqref{eq:associator-transport-series}.  In the Hodge
frames, topological variation on the two factors is given by
$2\pi i\,U(-a)$ and $2\pi i\,U(-b)$, respectively.
We use the tensor-product interval orientation of
\cite[Proposition~11]{MarkarianConvolution}, with $t$ increasing
from $0$ to $1$ on the exceptional component.
The tensor product of the Milnor complexes in the construction
of \eqref{eq:convolution-vanishing-betti} consequently sends
$\xi\in E_1\otimes E_2$ to the supported class on $(0,1)$ whose
horizontal representative is
\[
(2\pi i)^2G(t)U(-a)U(-b)\xi.
\]
This is the jump of a primitive representing the component
cohomology class.  If its de Rham representative is
$[q\,dt/(t(1-t))]$, the Cauchy calculation of
\cite[Proposition~5]{MarkarianHodge}, with $N_0=-b$ and
$N_1=a$, gives
\begin{equation}
q=2\pi i\,(1-a-b)
\left(\int_0^1G(t)\,dt\right)U(-a)U(-b)\xi.
\label{eq:beta-component-period}
\end{equation}
Indeed, that calculation divides the jump by $2\pi i$ and
then applies $(1-a-b)\int_0^1$.

The normal monodromy on this component is
$\exp(-2\pi i(a+b))$, as the factor $h^{-a-b}$ in the
specialization shows.  Passing from its inclusion in nearby
cycles to vanishing cycles by topological variation therefore
introduces $(2\pi i)^{-1}U(-a-b)^{-1}$ relative to the
logarithmic-variation frame \eqref{eq:exceptional-component-mhs}.
Indeed, in these de Rham frames and with the Tate lines
identified, $\var=2\pi i\,U(-a-b)\Var$ on the exceptional
image, whereas $\Var\circ\lambda=i_{\mathrm{exc}}(-1)$.
Thus
\begin{equation}
\begin{split}
\kappa_F^{-1}c_{\mathbb E\ast\mathbb F}\kappa_B
 (c_{\mathbb E}\otimes c_{\mathbb F})^{-1}
&=(1-a-b)\left(\int_0^1t^{-b}(1-t)^{-a}\,dt\right)
 \frac{U(-a)U(-b)}{U(-a-b)}\\
&=\frac{\Gamma(1-a)\Gamma(1-b)}{\Gamma(1-a-b)}
 \frac{U(-a)U(-b)}{U(-a-b)}.
\end{split}
\label{eq:beta-variation-factors}
\end{equation}
Euler's reflection formula gives
\[
U(u)=\frac{e^{\pi i u}}{\Gamma(1+u)\Gamma(1-u)}.
\]
Substitution in \eqref{eq:beta-variation-factors} gives
$\Gamma(1+a+b)/(\Gamma(1+a)\Gamma(1+b))$.
Taking the inverse proves
\eqref{eq:convolution-relative-period}; Euler's beta integral
gives the integral in \eqref{eq:beta-factor}.
Since $a,b$ are nilpotent, all matrix powers are finite sums
of products of $\log t$ and $\log(1-t)$.  Such products are
integrable on $(0,1)$, so neither integral requires
regularization.
Since $a,b$ lower weight and $\mathcal B(0,0)=1$, formula
\eqref{eq:convolution-relative-period} also shows that $\kappa_B$
preserves $W$ and that its associated graded map agrees with
that of $\kappa_F$ under Betti--de Rham comparison.
\end{proof}

To obtain a period matrix in rational frames, choose weight
framings of $\Phi_1(\mathbb E)$ and $\Phi_1(\mathbb F)$ and
transport their tensor product by $\kappa_B$.
With the period convention of
\cite[Subsection~4.1]{MarkarianHodge}, formula
\eqref{eq:convolution-relative-period} says precisely
\begin{equation}
P_{\Phi_1(\mathbb E\ast\mathbb F)}
=\bigl(P_{\Phi_1(\mathbb E)}\otimes P_{\Phi_1(\mathbb F)}\bigr)
 \mathcal B(a,b).
\label{eq:convolution-vanishing-period-matrix}
\end{equation}
In particular, the new factor depends only on the local
logarithmic monodromies at $1$.  Its coefficients are given by
\[
\mathcal B(a,b)=\exp\left(
\sum_{n\geq2}\frac{(-1)^n\zeta(n)}{n}
\bigl(a^n+b^n-(a+b)^n\bigr)\right).
\]
For example, if $a^2=b^2=0$, then
$\mathcal B(a,b)=1-\zeta(2)ab$.

\subsection{Hodge defect of the convolution}
\label{subsec:convolution-hodge-defect}

Write $H(\mathbb E)=H^0(\mathbb P^1,\mathbb E)$, and similarly
for $\mathbb F$ and their convolution.  Thus these are the
first cohomology groups of the unshifted extensions.
Let
\[
K:H(\mathbb E\ast\mathbb F)
\xrightarrow{\sim}H(\mathbb E)\otimes H(\mathbb F)
\]
be the isomorphism of Proposition~\ref{prop:convolution-cohomology}.
The compatibility theorem
\cite[Theorem~2, Subsection~4.2]{MarkarianConvolution} states
that the following square of Betti realizations commutes:
\begin{equation}
\begin{tikzcd}[cramped, column sep=large, row sep=large]
\Phi_1(\mathbb E)_B\otimes\Phi_1(\mathbb F)_B
 \arrow[r,"\kappa_B"]
 \arrow[d,"\varphi_I\otimes\varphi_I"']
&\Phi_1(\mathbb E\ast\mathbb F)_B\arrow[d,"\varphi_I"]\\
H(\mathbb E)_B\otimes H(\mathbb F)_B
 \arrow[r,"K_B^{-1}"']
&H(\mathbb E\ast\mathbb F)_B.
\end{tikzcd}
\label{eq:convolution-shrinking-square}
\end{equation}
Since $K$ respects the mixed Hodge structures, its de Rham
realization also satisfies
\[
K_{\mathrm{dR}}\rho_{\mathbb E\ast\mathbb F}\kappa_F
=\rho_{\mathbb E}\otimes\rho_{\mathbb F}.
\]
Here $\rho$ is the residue isomorphism
\eqref{eq:residue-pair-class}; restrictions to $\mathbb C^*$
are understood in this notation.
Indeed, both sides preserve the Hodge splittings and induce
the same map on the associated graded of
\eqref{eq:convolution-shrinking-square}.
Combining this equality with
\eqref{eq:defect-dr-calculation} and
\eqref{eq:convolution-relative-period} gives
\begin{equation}
\kappa_F^{-1}\mathcal G_I(\mathbb E\ast\mathbb F)\kappa_F
=\bigl(\mathcal G_I(\mathbb E)\otimes
       \mathcal G_I(\mathbb F)\bigr)\mathcal B(a,b).
\label{eq:convolution-defect-identity}
\end{equation}
This compares the tensor-product Hodge structure supplied by
$K$ with the one calculated from vanishing cycles.

To express \eqref{eq:convolution-defect-identity} in the
transport algebra, we use the Deligne--Terasoma generators
\cite{DeligneTerasoma}; see also
\cite[Subsection~1.2, formula~(1.2.1)]{EnriquezFurusho}:
\begin{equation}
y_n=-e_0^{n-1}e_1\quad(n\geq1),\qquad y_0=1,
\qquad
\Delta_*(y_n)=\sum_{i=0}^n y_i\otimes y_{n-i}.
\label{eq:harmonic-coproduct}
\end{equation}
The $y_n$ freely generate $W$ as a complete algebra, and this
formula defines its harmonic coproduct
$\Delta_*:W\to W\widehat\otimes W$.
Under the change of letters $x_0=e_0$, $x_1=-e_1$, these are
Racinet's generators $x_0^{n-1}x_1$
\cite[Subsection~2.2.5]{Racinet}.

\begin{lemma}
\label{lem:convolution-transport-coproduct}
In the residue frame $\kappa_F$, the action of the transport
algebra on vanishing cycles of the convolution is
\begin{equation}
\kappa_F^{-1}R_{\mathbb E\ast\mathbb F}(w)\kappa_F
=(R_{\mathbb E}\otimes R_{\mathbb F})(\Delta_*(w)).
\label{eq:convolution-transport-coproduct}
\end{equation}
\end{lemma}

\begin{proof}
We give the logarithmic calculation underlying
\cite[Theorem~3]{MarkarianConvolution}, keeping track of its
Hodge frame.  Use coordinates $(u,z/u)$ on a multiplication
fiber, and put
\[
a_0=N_0^{(1)}\otimes1,\qquad b_0=1\otimes N_0^{(2)}.
\]
As before, $a=N_1^{(1)}\otimes1$ and $b=1\otimes N_1^{(2)}$.
Replace the $!$-extension at $u=0$ by the $*$-extension.
The cone is supported on this boundary and its direct image
is smooth near $z=1$, so this replacement does not change
vanishing cycles at $1$.

The relative logarithmic complex now has $*$ conditions at
$u=0,1,z$ and a $!$ condition at $\infty$.
Its degree-zero term is a sum of copies of
$\mathcal O_{\mathbb P^1}(-\infty)$, which has no cohomology.
Consequently its first cohomology has the residue frame
\[
\alpha(p,q)=p\eta_z-q\eta_1,\qquad
\eta_s=\frac{du}{u-s}-\frac{du}{u},
\qquad p,q\in E_1\otimes E_2.
\]
This frame respects the Hodge filtration.  The total connection
is $d-\Omega$, where
\[
\Omega=(a_0-b_0+b)\frac{du}{u}
-a\frac{du}{u-1}-b\frac{du}{u-z}
+b_0\frac{dz}{z}+b\frac{dz}{u-z}.
\]
Differentiating $\alpha(p,q)$ in $z$ and reducing modulo the
relative exact term $\nabla_u(-p/(u-z))$ gives the connection
\begin{equation}
d-
\begin{pmatrix}a_0&b\\0&b_0\end{pmatrix}\frac{dz}{z}
-
\begin{pmatrix}a&b\\a&b\end{pmatrix}\frac{dz}{1-z}.
\label{eq:convolution-residue-matrices}
\end{equation}
The primitive used in this reduction vanishes at $\infty$;
its pole at the $*$ divisor $u=z$ is allowed in the localized
de Rham complex.

In the coordinates of Subsection~\ref{subsec:beta-factor},
$u=z/(1-ht)$, and on each fiber
\[
\eta_z-\eta_1=\frac{dt}{t(1-t)}.
\]
Thus logarithmic variation of the class defining $\kappa_F$
is the column $v=(1,1)^t$ in the frame $\alpha$.
The relation $vc=-N_1$ then gives $c=-(a,b)$.
Formula \eqref{eq:dr-transport-action} yields
\begin{align*}
R_{\mathbb E\ast\mathbb F}(y_n)
&=(a,b)\begin{pmatrix}a_0&b\\0&b_0\end{pmatrix}^{n-1}
  \binom11\\
&=aa_0^{n-1}+bb_0^{n-1}
  +\sum_{k+l=n-2}(aa_0^k)(bb_0^l).
\end{align*}
The sum is empty for $n=1$.
Since operators on different tensor factors commute, this is
exactly the action of \eqref{eq:harmonic-coproduct}.
The assertion for all $w$ follows by multiplication and
completion.
\end{proof}

To detect coefficients, we use finite quotients of the
universal pro-unipotent path variation.  Fix $z_*=1/2$ and
let $P(z_*,z)$ be the set of homotopy classes of paths from
$z_*$ to $z$ in $X$.  If $J$ is the augmentation ideal of
$\mathbb Q[\pi_1(X,z_*)]$, put
\[
\mathbb U_{m,B,z}
=\mathbb Q[P(z_*,z)]/\mathbb Q[P(z_*,z)]J^{m+1}.
\]
The right action is composition with loops at $z_*$.
The inverse system $\mathbb U=\varprojlim_m\mathbb U_m$
is the linearized pro-unipotent universal cover;
we use only its finite quotients.

\begin{lemma}[Universal Kummer variations]
\label{lem:universal-kummer-variations}
Each $\mathbb U_m$ is an admissible variation of mixed
Hodge--Tate structures with constant Tate graded pieces.
Its de Rham realization is the trivial bundle with fiber
\[
T_m=\mathbb C\langle e_0,e_1\rangle/(e_0,e_1)^{m+1}
\]
and connection
\[
\nabla=d-e_0\frac{dz}{z}-e_1\frac{dz}{z-1},
\]
where $e_0,e_1$ act by left multiplication.
For the $*$-extension $j_{1*}\mathbb U_m[1]$, the transport
action satisfies, for $w\in W$,
\[
R(w)(1)=\operatorname{rev}(w)\pmod{(e_0,e_1)^{m+1}}.
\]
Consequently, evaluations on pairs $\mathbb U_m,\mathbb U_n$
detect every coefficient in $W\widehat\otimes W$.
\end{lemma}

\begin{proof}
This is the canonical path variation of
\cite[Sections~6--7]{CarlsonHain}.  We recall its realization
in the present case.  Truncated holonomy for the displayed
connection maps the path module to $T_m$.  Positive loops
around $0$ and $1$ have linear terms $2\pi i e_0$ and
$2\pi i e_1$.  Since $\pi_1(X,z_*)$ is free on these two loops,
the resulting comparison is an isomorphism on the
augmentation-graded spaces, and hence on every finite quotient.
This supplies the rational Betti realization of the displayed
connection.

Writing $|w|$ for word length, its filtrations are
\[
W_{-2r}T_m=\langle w:|w|\ge r\rangle,
\qquad F^{-r}T_m=\langle w:|w|\le r\rangle.
\]
Thus
$\operatorname{gr}_{-2r}^W\mathbb U_m
\simeq\mathbb Q_X(r)^{\oplus2^r}$ for $0\le r\le m$.
The connection increases word length by one, giving Griffiths
transversality.  Its residues lower $W$ by two, so $W$ itself
is the relative monodromy filtration at each boundary point.
Both filtrations extend to the trivial logarithmic bundle
with locally free double-graded quotients.  This verifies
admissibility and the Hodge--Tate assertion.

In the notation of \eqref{eq:logarithmic-connection},
$N_0$ is left multiplication by $e_0$ and $N_1$ is left
multiplication by $-e_1$.  Logarithmic variation for the
$*$-extension gives $v=1$ and $c=-N_1$.
Formula \eqref{eq:dr-transport-action} therefore makes
$R(w)$ left multiplication by $\operatorname{rev}(w)$.
Application to $1$ proves the displayed identity.
On pairs, application to $1\otimes1$ detects all words of
bidegree at most $(m,n)$.  Increasing $m,n$ detects every
coefficient of the completed tensor product.
\end{proof}

\begin{theorem}[Double shuffle relations]
\label{thm:double-shuffle}
The KZ associator $A$ defined in
Subsection~\ref{subsec:defect-transport} satisfies
\begin{equation}
\mathcal B(y_1\otimes1,1\otimes y_1)\,
\Delta_*(A_Y)=A_Y\otimes A_Y.
\label{eq:double-shuffle}
\end{equation}
Together with the group-like property of $A$ for the coproduct
in which $e_0,e_1$ are primitive, this is the beta-factor form
of Racinet's regularized double shuffle relations
\cite[Definition~3.2.1]{Racinet}; see also
\cite[Subsection~1.5]{EnriquezFurusho}.
\end{theorem}

\begin{proof}
Apply \eqref{eq:convolution-defect-identity} to the
$j_{**!}$-extensions of $\mathbb U_m$ and $\mathbb U_n$.
Substitute Theorem~\ref{thm:general-defect-transport} and
Lemma~\ref{lem:convolution-transport-coproduct}.
Because $R$ is a right action, the order of multiplication
is reversed when it is written as a column operator.
Evaluation on $1\otimes1$ and the coefficient detection of
Lemma~\ref{lem:universal-kummer-variations} give, as $m,n$ vary,
the identity
\begin{equation}
\Delta_*(g_I)=
\mathcal B(y_1\otimes1,1\otimes y_1)(g_I\otimes g_I).
\label{eq:double-shuffle-with-variation}
\end{equation}
Put $F=A(-e_0,-e_1)_Y$, so that $g_I=U(e_1)F$.
The element $y_1=-e_1$ is primitive for $\Delta_*$.
Substitution in \eqref{eq:double-shuffle-with-variation}, using
\[
\frac{U(-u-v)}{\mathcal B(u,v)U(-u)U(-v)}
=\mathcal B(-u,-v),
\]
therefore gives
\[
\mathcal B(-y_1\otimes1,-1\otimes y_1)\,
\Delta_*(F)=F\otimes F.
\]
Finally, $e_i\mapsto-e_i$ sends $y_n$ to $(-1)^ny_n$ and
preserves \eqref{eq:harmonic-coproduct}.  Applying this
involution gives \eqref{eq:double-shuffle}.
\end{proof}

Since $y_1$ is primitive, the exponential formula for
$\mathcal B$ rewrites \eqref{eq:double-shuffle} as the
group-like property, for $\Delta_*$, of the series
\[
\exp\left(\sum_{n\ge2}
\frac{(-1)^{n-1}\zeta(n)}{n}\,y_1^n\right)A_Y.
\]
This is Racinet's corrected series
\cite[Definition~3.2.1]{Racinet}.

To express the beta factor as an abelian evaluation, first apply
the projection $A\mapsto A_Y$ and then abelianize.  Denote the
result by $(A_Y)^{\mathrm{ab}}(x,y)$.  In our coordinates,
one obtains
\[
(A_Y)^{\mathrm{ab}}(x,y)
=\frac{\Gamma(1-x)\Gamma(1-y)}{\Gamma(1-x-y)},
\qquad xy=yx.
\]
Indeed, apply \cite[Proposition~6 and Theorem~1]{MarkarianHodge}
to commuting residues, using $dz/(z-1)=-dz/(1-z)$, and cancel
the variation factor $U$.
This is also the abelianization formula of
\cite[Lemma~9.5(1), formula~(9.2.1)]{EnriquezFurusho}.
Hence the correction factor in \eqref{eq:double-shuffle} is
\begin{equation}
(A_Y)^{\mathrm{ab}}(-y_1\otimes1,-1\otimes y_1)
=\mathcal B(y_1\otimes1,1\otimes y_1).
\label{eq:double-shuffle-abelian-factor}
\end{equation}

\bibliographystyle{alpha}
\bibliography{hodge}

\end{document}